\documentclass[aop,MSNbibl,citesort,dvips]{arximspdf}

\doi{10.1214/11-AOP701}
\volume{41}
\issue{2}
\pubyear{2013}
\firstpage{699}
\lastpage{721}

\makeatletter
\def\eps{\varepsilon}
\def\flow{\varphi}
\newcommand{\const}[1]{c^{\scriptscriptstyle(#1)}}
\newcommand{\tildconst}[1]{\tilde c^{\scriptscriptstyle(#1)}}
\def\R{\mathbf{R}}
\def\N{\mathbf{N}}
\def\dx{\mathrm{d}}
\newcommand{\Lipschitz}{\operatorname{Lip}}
\def\Lip{\Lipschitz_0}
\newcommand{\diam}{\operatorname{diam}}
\newcommand{\id}{\operatorname{id}}
\newcommand{\rref}[1]{(\ref{#1})}
\newcommand{\notag}{\nonumber}
\newtheorem{theorem}{Theorem}[section]
\newtheorem{lemma}{Lemma}[section]
\newproclaim{remark}{Remark}
\newcommand{\fracc}[2]{{#1}/{(#2)}}
\newcommand{\fraca}[2]{{#1}/{#2}}
\newcommand{\fracb}[2]{{(#1)}/{#2}}

\makeatother

\begin{document}
\begin{frontmatter}

\title{Asymptotic support theorem for planar isotropic Brownian flows\thanksref{TT1}}
\runtitle{Asymptotic support theorem for planar IBF}

\begin{aug}
\author[A]{\fnms{Moritz} \snm{Biskamp}\corref{}\ead[label=e1]{biskamp@math.tu-berlin.de}}
\runauthor{M. Biskamp}
\affiliation{Technische Universit\"at Berlin}
\address[A]{Technische Universit\"at Berlin\\
Institut f\"ur Mathematik, MA 7-4\\
Stra{\ss}e des 17. Juni 136\\
10623 Berlin\\
Germany\\
\printead{e1}} 
\end{aug}
\thankstext{TT1}{Supported by the International Research Training
Group \textit{Stochastic Models of Complex Processes} funded by the German
Research Council (DFG).}

\received{\smonth{7} \syear{2010}}
\revised{\smonth{6} \syear{2011}}

%
\begin{abstract}
It has been shown by various authors that the diameter of a given
nontrivial bounded connected set $\mathcal{X}$
grows linearly in time under the action of an isotropic Brownian flow
(IBF), which has a nonnegative top-Lyapunov exponent.
In case of a planar IBF with a positive top-Lyapunov exponent, the
precise deterministic linear growth rate $K$ of the diameter
is known to exist. In this paper we will extend this result to an
asymptotic support theorem for the time-scaled trajectories of
a planar IBF $\varphi$, which has a positive top-Lyapunov exponent,
starting in a nontrivial compact connected set
$\mathcal{X} \subseteq\R^2$; that is, we will show convergence in
probability of the set of time-scaled trajectories
in the Hausdorff distance to the set of Lipschitz continuous functions
on $[0,1]$ starting in $0$ with Lipschitz constant $K$.
\end{abstract}

%

\end{frontmatter}

\section{Introduction}

Isotropic Brownian flows (IBFs) are a fairly natural class of
stochastic flows and were first introduced by It\^o~\cite{ito56} and
Yaglom~\cite{yag57}. For this class of stochastic flows, the image of
a single point is a Brownian motion, and the covariance tensor between
two different Brownian motions is an isotropic function of their
positions. IBFs, and in particular their local structure, have been
extensively studied in the 1980s by~\cite{jan85} and~\cite{bax86},
among others.

The study of the global behavior of stochastic flows was stimulated by
Carmona's conjecture~\cite{cc99}, Section 5.2., that the diameter of
the image of a compact set could expand linearly in time, but not
faster. For stochastic flows this conjecture was proved by Cranston,
Scheutzow and Steinsaltz~\cite{css00} and improved by Lisei and
Scheutzow~\cite{ls01} as well as by Scheutzow~\cite{sch09}. Even
more surprising than this upper bound is maybe the existence of points
that move with linear speed, although each individual point as a
diffusion grows on average, like the square-root of the time. This
lower bound was proved first for IBFs, which have a strictly positive
top-Lyapunov exponent, by Cranston, Scheutzow and Steinsaltz \cite
{css99} and under more general conditions by Scheutzow and Steinsaltz
\cite{ss02}. Nevertheless, upper and lower bounds for the linear
growth turn out to be far from each other in some examples. In the case
of planar periodic stochastic flows (stochastic flows on the torus)
Dolgopyat, Kaloshin and Koralov~\cite{dkk04} used a new approach
based on the so-called \textit{stable norm}, to identify the precise
deterministic linear growth rate of such flows. By this approach, van
Bargen~\cite{vB10} identified the precise deterministic growth rate
for planar IBFs, which have a strictly positive top-Lyapunov exponent.

Not only has the linear growth rate been analyzed in the last years,
but also the behavior of the individual trajectories of stochastic
flows. Scheutzow and Steinsaltz~\cite{ss02} investigated so-called
\textit{ball-chasing} properties of the flow, which is the existence
of a
trajectory that follows a given Lipschitz path in a logarithmic
neighborhood~\cite{ss02}, Theorem 4.2, where the Lipschitz constant
is basically the lower bound of linear growth mentioned in the previous
paragraph.

Here we are looking at the individual trajectories of a planar IBF, or,
to be more precise, at the linear time-scaled versions. Getting a
better understanding of these trajectories yields a deeper
understanding of the expansion of nontrivial bounded connected sets
under the action of an IBF. In this paper we will show convergence in
probability of the set of time-scaled trajectories in the Hausdorff
distance to the set of Lipschitz continuous functions starting in $0$
with Lipschitz constant $K$, which is the deterministic growth rate for
a planar IBF mentioned above. Roughly speaking we will show the following:
On the one hand, for any time-scaled trajectory, there exists a
Lipschitz function with Lipschitz constant $K$ starting in $0$ such
that this function is close to the time-scaled trajectory. This yields
an upper bound on the speed of the trajectories. Hence we will call
this inclusion the \textit{upper bound}.
On the other hand we show that for any given Lipschitz function with
Lipschitz constant $K$ starting in $0$, there exists a trajectory
that approximates this Lipschitz function. This gives a lower bound on
the maximum speed of the trajectories. Thus we will refer to this
inclusion as the \textit{lower bound}.
As far as the author knows such a complete characterization of the
asymptotic behavior of the~trajectories of stochastic flows is a
novelty in the present context and hence yields a new and deeper
understanding of the expansion of nontrivial bounded connected sets
under the action of IBFs.

The paper is organized as follows: In Section~\ref{sec:IBF} we first
introduce the notion of stochastic flows, and in particular of IBFs and
some of their main properties used within this paper. After stating the
main theorem in Section~\ref{sec:mainthm}, we first introduce the
notion of \textit{stable norm} in Section~\ref{sec:stablenorm}. The proof
of the main theorem is divided into the proof of the upper bound
(Section~\ref{sec:upperBound}) and the lower bound (Section \ref
{sec:lowerBound}).

\section{Preliminaries}
\subsection{Isotropic Brownian flows} \label{sec:IBF}
We provide a short introduction to isotropic Brownian flows (IBF)
following mainly~\cite{bax09}.

A \textit{stochastic flow of homeomorphisms} on $\R^d$ is a family of
random homeomorphisms $\{\flow_{s,t}\dvtx s, t \in\R_+\}$ of $\R^d$,
which almost surely satisfies the flow property, that is, $\flow_{s,t}
= \flow_{u,t} \circ\flow_{s,u}$ for all $s,t,u \in\R_+$, and
$\flow_{t,t} =\id|_{\R^d}$ for all $t \in\R_+$, and is jointly
continuous; that is, $(s,t,x) \mapsto\flow_{s,t}(x)$ is continuous.
The flow is called a \textit{Brownian flow} if the increments $\flow
_{s,t}$ on disjoint intervals are independent and time homogeneous.

Due to~\cite{kun90}, Theorem 4.2.8, under suitable regularity
conditions, Brownian flows of homeomorphisms can be realized as
solutions of Kunita-type stochastic differential equations
\begin{eqnarray*}
\flow_{s,t}(x) = x + \int_s^t M(\dx u, \flow_{s,u}(x)) + \int_s^t
v(\flow_{s,u}(x)) \,\dx u, \qquad s \leq t,
\end{eqnarray*}
where $v\dvtx \R^d \to\R^d$ is a vector field, and $M\dvtx \R_+
\times\R
^d \times\Omega\to\R^d$ is a mean-zero Gaussian martingale field on
a complete probability space $(\Omega,\mathcal{F},\mathbf{P})$. $M$
is called the generating \textit{Brownian field} and its distribution is
determined by the covariances
\begin{eqnarray*}
\mathbf{E}[\langle M(t,x),\xi\rangle\langle M(s,y),\eta\rangle] = (s
\wedge t)\langle b(x,y)\xi,\eta\rangle,\qquad\xi,\eta\in\R^d,
\end{eqnarray*}
where $b\dvtx \R^d \times\R^d \to\R^{d \times d}$ is a covariance
tensor. The distribution of the flow $\{\flow_{s,t}\dvtx s,t\in\R_+\}
$ is
determined by the functions $b(x,y)$ and $v(x)$. If there is no risk of
ambiguity, we will abbreviate $\flow_{0,t}(x)$ by $\flow_t(x)$. Due
to the independent increments and the flow property, a Brownian flow
satisfies, according to~\cite{kun90}, Theorem 4.2.1, a Markov
property in the following sense: Let $\mathcal{F}_{s,t}$ be the least
sub $\sigma$-algebra of $\mathcal{F}$ containing all null sets and
$\bigcap_{\eps> 0}\{\flow_{u,r} \dvtx s - \eps\leq u,r\leq t+\eps
\}$.
Then for $0 \leq s < t < u$, $n \in\N$ and $x_1,\ldots, x_n \in\R
^d$, we have
%
%
\begin{eqnarray} \label{eq:markovprop}
&&\mathbf{P}\bigl( (\flow_{s,u}(x_1),\ldots,\flow_{s,u}(x_n)
)\in E | \mathcal{F}_{s,t}\bigr)\nonumber
\\[-8pt]
\\[-8pt]
&& \qquad = \mathbf{P}\bigl( (\flow_{t,u}(y_1),\ldots,\flow
_{t,u}(y_n) )\in E\bigr) |_{y_i=\flow_{s,t}(x_i)},
\nonumber
\end{eqnarray}
where $E$ is a Borel sets in $\R^{nd}$.

An \textit{isotropic Brownian flow} on $\R^d$ is a Brownian flow of
homeomorphisms of $\R^d$, where the distribution of each $\flow
_{s,t}$ is invariant under rigid transformations of $\R^d$. The
invariance in distribution of $\flow_{s,t}$ under rigid motions
implies the invariance in distribution of the generating Brownian field
$M(t,x)$; in this case $M(t,x)$ is said to be an \textit{isotropic
Brownian field}. The invariance under translations implies that $b(x,y)
= b(x-y,0) \equiv b(x-y)$, and then the invariance under rotations and
reflections implies that
%
%
\begin{eqnarray} \label{eq:isotropy}
b(x) = O^T b(Ox)O
\end{eqnarray}
for all orthogonal matrices $O$ on $\R^d$. Moreover we have $v(x)
\equiv0$. In this paper we will assume $b \in C^\infty$, since we
will use results of~\cite{vB10}, where smoothness of $b$ has to be
assumed. In this case $\flow_{s,t}(\cdot) \in C^\infty(\R^d)$ are
diffeomorphims. Furthermore, the isotropy property \rref{eq:isotropy}
implies that $b(0) = c\id|_{\R^d}$ for some constant $c>0$. At the
cost of rescaling time by a constant factor, we can and will assume
that $b(0) = \id|_{\R^d}$. In order to avoid the trivial case where
the flow consists of translations, we assume also that $b(x) \not
\equiv\id|_{\R^d}$. Since the properties of the flow we are
interested in do not depend on rigid translations of the space by a
Brownian motion added to the generated IBF, we can and will assume that
$\lim_{| x|\to\infty}b(x) = 0$.

According to~\cite{yag57}, Section 4 (and as described in \cite
{bax86}), a covariance tensor with the above properties can be written
in the form
\begin{eqnarray*}
b_{ij}(x) =
\cases{\displaystyle
\bigl(B_L(| x|) - B_N(| x|)\bigr) \frac{x_ix_j}{| x|^2} + \delta
_{ij}B_N(| x|)
,&\quad if $x \neq0$,\cr\displaystyle
\delta_{ij} ,&\quad if $x = 0$
}
\end{eqnarray*}
for $i,j = 1, \ldots, d$, where $B_L$ and $B_N$ are the so-called
longitudinal and transverse (normal) covariance functions defined by
\begin{eqnarray*}
B_L(r) := b_{ii}(re_i), \qquad B_N(r) := b_{ii}(re_j)
\end{eqnarray*}
for $r \geq0$ and $i \neq j$, where $e_i$ denotes the $i$th unit
vector in $\R^d$. For future reference define
\begin{eqnarray*}
\beta_L := -B''_L(0) > 0, \qquad\beta_N := -B''_N(0) > 0,
\end{eqnarray*}
to be the negative second right-hand derivative of the longitudinal and
respectively transverse covariance function. In Lemma~\ref{lem:kappa}
we will give an estimate of the longitudinal and transverse covariance
functions in terms of $\beta_L$ and $\beta_N$, respectively.

Lyapunov exponents can be defined for dynamical systems and
characterize the exponential rate of separation of infinitesimally
close trajectories. Baxendale and Harris~\cite{bax86} have shown
under the assumptions mentioned above, that IBFs have Lyapunov
exponents, which satisfy
\begin{eqnarray*}
\mu_i = \tfrac{1}{2} \bigl((d-i)\beta_N - i\beta_L \bigr), \qquad
i=1, \ldots, d.
\end{eqnarray*}
The top-Lyapunov exponent $\mu_1$, and more precisely its sign,
crucially affects the asymptotic behavior of the flow. As shown in
\cite{css99} and~\cite{ss02}, a nonnegative top Lyapunov exponent
$\mu_1 \geq0$ implies that any nontrivial bounded set (a set is said
to be \textit{nontrivial} if it is connected and contains more than one
point) does not contract to a single point under the action of the
flow. On the other hand,\vadjust{\goodbreak} if $\mu_1 < 0$, then according to \cite
{ss02} there is a positive probability that a small set contracts to a
single point, and hence our result cannot be true. By this remark, and
since we would like to use results from~\cite{vB10}, we always will
assume a strictly positive top-Lyapunov exponent. But we conjecture
that the results in~\cite{vB10}, and hence our main result, are also
true for $\mu_1 = 0$.
For more details on Lyapunov exponents for random dynamical systems we
refer to~\cite{arn98}.

If the flow $\flow$ is restricted to $\{(s,t)\in\R_+ \times\R_+
\dvtx
s \leq t\}$, it is called the \textit{forward flow}, whereas if restricted
to $\{(s,t)\in\R_+ \times\R_+ \dvtx s \geq t\}$, it is called the
\textit{backward flow}. In Kunita~\cite{kun90}, Theorem 4.2.10, the
generating Brownian field of the backward Brownian flow has been
calculated. If the flow is isotropic it turns out that it is in fact
equal to the generating Brownian field of the forward Brownian flow;
see~\cite{bax86}, (3.7). This implies that for fixed $T>0$, we have
%
%
\begin{eqnarray} \label{eq:timereversal}
\mathcal{L}[\flow_{s,t}(\cdot) \dvtx 0 \leq s \leq t \leq T] =
\mathcal
{L}[\flow_{T-s,T-t}(\cdot) \dvtx 0 \leq s \leq t \leq T],
\end{eqnarray}
the so-called \textit{time reversal} property of IBFs.

\subsection{Time-scaled trajectories} \label{sec:trajectories}

Let $\mathcal{X} \subseteq\R^d$ be compact, and denote the \textit{set
of time-scaled trajectories of the flow starting in $\mathcal{X}$} up
to some time $T > 0$ by
\begin{eqnarray*}
F_T(\mathcal{X},\omega) :=\bigcup_{x \in\mathcal{X}} \biggl\{[0,1]
\ni t \mapsto\frac{1}{T}\flow_{0,tT}(x,\omega) \biggr\}
\end{eqnarray*}
for $\omega\in\Omega$. Since $\mathcal{X}$ is compact, and $(x,t)
\mapsto\flow_{0,t}(x)$ is continuous, we have that $F_T(\mathcal
{X})$ is a compact subset of the continuous functions on $[0,1]$ with
respect to the supremum norm $\|\cdot\|_\infty$. Further denote by
$\Lip(K)$ the set of Lipschitz continuous functions $f$ on $[0,1]$
with $f(0) = 0$ and Lipschitz constant $K$, which is as well a compact
set with respect to $\|\cdot\|_\infty$. The Hausdorff distance
between two nonempty compact sets $A$ and $B$ of a metric space is
defined by
\begin{eqnarray*}
d_H(A,B) := \max\Bigl\{\sup_{x\in A} d(x,B) ; \sup_{y \in B} d(y,
A) \Bigr\},
\end{eqnarray*}
where $d$ denotes the metric. Since $F_T(\mathcal{X})$ and $\Lip(K)$
are compact subsets of $C([0,1],\|\cdot\|_\infty)$, the function
\begin{eqnarray*}
(T,\omega) \mapsto d_H(F_T(\mathcal{X},\omega), \Lip(K))
\end{eqnarray*}
is well defined.

\section{Main theorem} \label{sec:mainthm}

From here on we will consider the case of planar IBFs; that is, the
dimension of the space will be $d =2$. Given a planar IBF $\flow$,
which has a strictly positive top-Lyapunov exponent, our main result
is: For any nontrivial compact connected set $\mathcal{X} \subseteq
\R^2$, we have convergence in probability of $d_H(F_T(\mathcal{X}),
\Lip(K))$ to $0$ for $T \to\infty$, that is, the following theorem.\vadjust{\goodbreak}

%
\begin{theorem} \label{thm:mainthm}
Let $\flow$ be a planar IBF, which has a strictly positive
top-Lyapunov exponent. Then there exists a deterministic constant $K >
0$ such that for any $\eps> 0$ and any nontrivial compact and
connected set $\mathcal{X} \subseteq\R^2$, we have
\begin{eqnarray*}
\lim_{T \to\infty} \mathbf{P}\bigl(d_H(F_T(\mathcal{X}), \Lip
(K)) >
\eps\bigr) = 0,
\end{eqnarray*}
where $d_H$ denotes the Hausdorff distance, $F_T(\mathcal{X})$ the set
of time-scaled trajectories (see Section~\ref{sec:trajectories}) and
$\Lip(K)$ the set of Lipschitz continuous functions on $[0,1]$
starting in $0$ with Lipschitz constant $K$.
\end{theorem}

The theorem will be proved in Section~\ref{sec:proofOfThm}.

\section{Stable norm} \label{sec:stablenorm}
The concept of stable norm presented in this section traces back to
Dolgopyat, Kaloshin and Koralov~\cite{dkk04}, where they considered
planar periodic stochastic flows.

Denote by $B_r(w)$ the closed ball in $\R^2$ of radius $r$ around $w
\in\R^2$. For any $R \geq1$, let $\mathcal{C}_R$ be the set of all
connected compact large subsets of $\R^2$ fully contained in
$B_{2R}(0)$, where a set is called \textit{large} if its diameter is
greater or equal than $1$. For $v \in\R^2$, $\mathcal{X}
\subseteq\R^2$ and $s \geq0$, define the stopping time
\begin{eqnarray*}
\tau^R(\mathcal{X},v,s) := \inf\{t \geq0 \dvtx \flow
_{s,s+t}(\mathcal{X}) \cap B_R(v) \neq\varnothing; \diam(\flow
_{s,s+t}(\mathcal{X}))\geq1 \},
\end{eqnarray*}
which is the first time when, starting at time $s$, the initial set
$\mathcal{X}$ under the action of the flow hits an $R$-neighborhood of
$v$ as a large set. For $s=0$, we will abbreviate in the following:
$\tau^R(\mathcal{X},v,0)$ by $\tau^R(\mathcal{X},v)$. By temporal
homogeneity of the flow, the laws of $\tau^R(\mathcal{X},v,s)$ and
$\tau^R(\mathcal{X},v)$ coincide. If only the distribution matters,
we will use $\tau^R(\mathcal{X},v)$. Then it is known from~\cite
{vB10} that for $v \in\R^2$ the following limit (uniformly in
$\mathcal{X} \in\mathcal{C}_R$) exists:
\begin{eqnarray*}
\|v\|^R := \lim_{t \to\infty} \frac{1}{t}\sup_{\gamma\in\mathcal
{C}_R} \mathbf{E}[\tau^R(\gamma,vt)] = \lim_{t \to\infty} \frac
{1}{t} \mathbf{E}[\tau^R(\mathcal{X},vt)].
\end{eqnarray*}
This limit is called the \textit{stable norm} of $v$. Further it is known
that $\|\cdot\|^R$ does not depend on the precise choice of $R \geq
1$, and it is indeed a norm on $\R^2$; see~\cite{vB10}, Section
3.2.2. Hence for the sequel, fix some arbitrary $R \geq1$. If
we denote the closed unit ball in $\R^2$ with respect to $\|\cdot\|
^R$ by $\mathcal{B}$, then, as shown by van Bargen~\cite{vB10},
Theorem 2.1, for any $\eps> 0$ and any nontrivial bounded connected
$\mathcal{X} \subseteq\R^2$,
%
%
\begin{eqnarray} \label{thm:holger1}
\lim_{T \to\infty}\mathbf{P}\biggl((1-\eps)T\mathcal{B}
\subseteq
\bigcup_{x \in\mathcal{X}} \bigcup_{0 \leq t \leq T} \flow_t(x)
\subseteq(1+\eps)T\mathcal{B}\biggr) = 1.
\end{eqnarray}
For our purpose this immediately implies that for $\eps> 0$ and $t \in
(0,1]$, we have
%
%
\begin{eqnarray} \label{cor:upperbound}
\lim_{T\to\infty}\mathbf{P}\bigl(\flow_{t T}(\mathcal{X})
\subseteq t
T (1 + \eps) \mathcal{B}\bigr) = 1.
\end{eqnarray}
Since the flow is isotropic, $\mathcal{B}$ is a ball in $\R^2$ with
(Euclidean) radius $K$, that is, $K = 1/\|e_1\|^R > 0$. This
deterministic constant $K$ is the Lipschitz constant in Theorem~\ref{thm:mainthm}.

In the sequel we will need the following lemma from~\cite{vB10} on
convergence in probability of the time-scaled hitting time to the
stable norm.

%
\begin{lemma} \label{lem:holger2}
For any $\eps> 0$ and $v \in\R^2$, we have
\begin{eqnarray*}
\lim_{T \to\infty} \sup_{\gamma\in\mathcal{C}_R} \mathbf
{P}\biggl(\biggl|
\frac{\tau^R(\gamma, Tv)}{T} - \|v\|^R\biggr| > \eps\biggr) = 0.
\end{eqnarray*}
Moreover for any $m \in\N$, there exists a constant $\const{1}_m$
such that
\begin{eqnarray*}
\sup_{\gamma\in\mathcal{C}_R} \mathbf{P}\bigl(\tau^R(\gamma,
Tv) >
(\|v\|^R + \eps)T\bigr) \leq\const{1}_m T^{-m}.
\end{eqnarray*}
\end{lemma}

\begin{pf}
See~\cite{vB10}, Corollary 4.7, and~\cite{vB10}, (3.27).
\end{pf}

The following lemma ensures that the diameter uniformly in $\gamma\in
\mathcal{C}_R$ under the action of the flow stays large after $\sqrt
{T}$ with high probability for $T$ large.

%
\begin{lemma} \label{lem:diameter}
For any $m \in\N$ there exists a constant $\const{2}_m$ such that
for $T$ large,
\begin{eqnarray*}
\sup_{\gamma\in\mathcal{C}_R} \mathbf{P}\Bigl(\inf_{s \geq\sqrt{T}}
\diam(\flow_s(\gamma)) < 1\Bigr) \leq\const{2}_mT^{-m}.
\end{eqnarray*}
\end{lemma}

\begin{pf}
Following the ideas of van Bargen~\cite{vB10}, (3.15) and (3.16), for
any $m \in\N$ there exists some constant $\tildconst{2}_m$ such that
for sufficiently small $\delta> 0$ and $n \in\N$ large, we have
\begin{eqnarray*}
\sup_{\gamma\in\mathcal{C}_R} \mathbf{P}(S_n(\gamma)) := \sup
_{\gamma\in\mathcal{C}_R} \mathbf{P}\Bigl(\mathop{\mathop{\inf
}_{s \in\N}}_{
s \geq\lfloor\sqrt{n}\rfloor} \diam(\flow_s(\gamma)) < \delta
n\Bigr) \leq\tildconst{2}_mn^{-m}.
\end{eqnarray*}
Similar to~\cite{sch09}, Lemma 6, for $x, y \in\R^2$, there exists a
Brownian motion $W$ such that we have almost surely
\begin{eqnarray*}
\inf_{0 \leq t \leq1} \|\flow_t(x) - \flow_t(y)\| \geq\|x-y\|\exp
\biggl(-\frac{\kappa}{2} + \sqrt{\kappa}\inf_{0 \leq t \leq1}
W_t \biggr),
\end{eqnarray*}
where, according to Lemma~\ref{lem:kappa}, which can be found in the
\hyperref[appm]{Appendix}, we have $\kappa:= \max\{\beta_L;\beta
_N\}$. For $\gamma
\in\mathcal{C}_R$ and any integer $k \geq\lfloor\sqrt{T}\rfloor$,
we choose on $S_{\lfloor T\rfloor}(\gamma)^c$ points $x^{(k)},
y^{(k)} \in\flow_k(\gamma)$ such that $\|x^{(k)} - y^{(k)}\| =
\delta k$. Hence we get for $m \in\N$ and $k$ large enough,
\begin{eqnarray*}
&&\sup_{\gamma\in\mathcal{C}_R}\mathbf{P}\Bigl( \inf_{k \leq t
\leq k+1} \diam(\flow_t(\gamma)) < 1 | S_{\lfloor T\rfloor
}(\gamma)^c\Bigr)\\
&& \qquad \leq\sup_{\gamma\in\mathcal{C}_R} \mathbf{P}\Bigl(
\inf_{0 \leq t \leq1} \|\flow_t\bigl(x^{(k)}\bigr) - \flow_t\bigl
(y^{(k)}\bigr)\| < 1
| S_{\lfloor T\rfloor}(\gamma)^c\Bigr)\\
&& \qquad \leq\mathbf{P}\biggl(\delta k \exp\biggl(-\frac{\kappa
}{2} + \sqrt{\kappa}\inf_{0 \leq t \leq1} W_t \biggr) < 1\biggr)\\
&& \qquad \leq\frac{2}{\sqrt{2 \pi}} (\delta k)^{1/2} \exp
\biggl(-\frac{ (\log(\delta k) )^2}{2\kappa} \biggr).
\end{eqnarray*}
Choosing $k$ such that $(\delta k)^{\log(\delta k)} \geq\delta k^m$,
we get
\begin{eqnarray*}
&&\sup_{\gamma\in\mathcal{C}_R}\mathbf{P}\Bigl( \inf_{k \leq t
\leq k+1} \diam(\flow_t(\gamma)) < 1 \big| S_{\lfloor T\rfloor
}(\gamma)^c\Bigr)\\
&& \qquad \leq\frac{2}{\sqrt{2 \pi}}(\delta k)^{1/2} \exp
\biggl(-\frac{\log(\delta k^m)}{2\kappa} \biggr) = \frac{2}{\sqrt
{2 \pi}}\delta^{\fracb{\kappa-1}{(2\kappa)}} k^{\fracb{\kappa-
m}{(2\kappa)}}.
\end{eqnarray*}
Then there exists a constant $\const{2}_m$ such that for $T$ large,
\begin{eqnarray*}
&&\sup_{\gamma\in\mathcal{C}_R} \mathbf{P}\Bigl(\inf_{s \geq
\sqrt{T}}
\diam(\flow_s(\gamma)) < 1\Bigr) \\
&& \qquad \leq\sum_{k \geq\lfloor\sqrt{T}\rfloor}
\sup_{\gamma\in\mathcal{C}_R} \mathbf{P}\Bigl( \inf_{k \leq t
\leq k+1} \diam(\flow_t(\gamma)) < 1 \big|S_{\lfloor T\rfloor
}(\gamma)^c\Bigr)  + \sup_{\gamma\in\mathcal{C}_R}\mathbf
{P}\bigl(S_{\lfloor T\rfloor}(\gamma)\bigr)\\
&& \qquad \leq\const{2}_m T^{-m},
\end{eqnarray*}
which completes the proof.
\end{pf}

\begin{remark*}
Observe that in the previous lemma uniform convergence in $\gamma\in
\mathcal{C}_R$ is only achieved because the sets in $\mathcal{C}_R$
are large.
\end{remark*}

\section{\texorpdfstring{Proof of Theorem \protect\ref{thm:mainthm}}{Proof of Theorem 3.1}} \label{sec:proofOfThm}

As usual we consider a planar IBF $\flow$, which has a strictly
positive top-Lyapunov exponent. The upper bound (Section~\ref
{sec:upperBound}) and the lower bound (Section~\ref{sec:lowerBound})
of Theorem~\ref{thm:mainthm} will be proved for large sets, that is,
the initial set $\mathcal{X}$ is assumed to be in $\mathcal{C}_R$ for
some arbitrary fixed $R \geq1$. The generalization to nontrivial
compact connected sets will be done in Section~\ref{sec:proof}, which
then completes the proof of Theorem~\ref{thm:mainthm}.
\subsection{Upper bound} \label{sec:upperBound}

This section is devoted to the proof of the upper bound of Theorem \ref
{thm:mainthm}, that is, the following theorem.

%
\begin{theorem} \label{thm:upperbound}
For any $\eps> 0$ and $\mathcal{X} \in\mathcal{C}_R$, we have
\begin{eqnarray*}
\lim_{T \to\infty}\mathbf{P}\Bigl(\sup_{g \in F_T(\mathcal{X})}d
(g, \Lip(K) ) > \eps\Bigr) = 0,
\end{eqnarray*}
where $K$ is the Euclidean radius of the stable norm unit ball; see
Section~\ref{sec:stablenorm}.
\end{theorem}

The proof of Theorem~\ref{thm:upperbound} is divided into several
steps. The main idea is to show that the time-scaled trajectories
behave like Lipschitz functions on some sufficiently small discrete
grid (Lemma~\ref{lem:growingBepsT}), and between two supporting points
large growth of the initial set does not occur (Lemma \ref
{lem:chaining}). For the first estimate we have to control trajectories
starting inside some linearly growing set, which extends the result of
Lemma~\ref{lem:holger2}, where the initial set has a fixed diameter.
The basic lemma to control this is the following.

%
\begin{lemma} \label{lem:convBepsT}
For all $\eps> 0$, $v \in\R^2$ and $0 < \tilde\eps\leq\frac{\eps
}{6\|e_1\|^R}$, we have
\begin{eqnarray*}
\lim_{T \to\infty} \mathbf{P}\biggl(\biggl| \frac{\tau
^R(B_{\tilde\eps
T}(0), vT)}{T} - \|v\|^R\biggr| > \eps\biggr) = 0.
\end{eqnarray*}
\end{lemma}

\begin{pf}
Since $B_{R}(0) \subset B_{\tilde\eps T}(0)$ for $T$ large, we have,
because of Lemma~\ref{lem:holger2},
\begin{eqnarray*}
&&\mathbf{P}\bigl(\tau^R(B_{\tilde\eps T}(0), vT) > (\|v\|^R + \eps
)T\bigr)\\
&& \qquad \leq\mathbf{P}\bigl(\tau^R( B_{R}(0), vT) > (\|v\|^R
+ \eps)T\bigr) \to0.
\end{eqnarray*}
According to~\cite{vB10}, Lemma 4.4, there exists a constant $\alpha>
0$ such that
%
%
\begin{eqnarray} \label{eq:p1}
\inf_{\gamma\in\mathcal{C}^*_R} \inf_{t \geq\alpha} \mathbf
{P}\bigl(\flow_t(\gamma)\cap\partial B_R(0) \neq\varnothing;
\diam(\flow
_t(\gamma)) \geq1\bigr) =: p_1 > 0,
\end{eqnarray}
where $\mathcal{C}^*_R$ denotes the set of all large connected subsets
$\gamma$ of $\R^2$ with $\gamma\cap\partial B_R(0) \neq\varnothing
$. Estimate \rref{eq:p1} basically tells that, given some extra time
$\alpha$ uniformly in $\gamma\in\mathcal{C}_R^*$, there is a
positive probability that $\flow_t(\gamma)$ will stay intersected
with $\partial B_R(0)$. By spatial homogeneity, the time reversal
property of IBFs [see \rref{eq:timereversal}] and \rref{eq:p1}, we get
\begin{eqnarray*}
&&\mathbf{P}\bigl(\flow_{t+\alpha}( B_R(0)) \cap B_{\tilde\eps T}(vT)
\neq\varnothing\bigr) \\
&& \qquad = \mathbf{P}\bigl(B_R(vT) \cap\flow
_{t+\alpha
}(B_{\tilde\eps T}(0)) \neq\varnothing\bigr)\\
&& \qquad \geq\mathbf{P}\bigl(B_R(vT) \cap\flow_{t+\alpha
}(B_{\tilde\eps T}(0)) \neq\varnothing | \tau^R(B_{\tilde\eps
T}(0),vT) \leq t\bigr)\\ && \qquad \quad {}\cdot\mathbf{P}\bigl
(\tau^R(B_{\tilde
\eps T}(0), vT) \leq t\bigr)\\
&& \qquad \geq p_1 \mathbf{P}\bigl(\tau^R(B_{\tilde\eps T}(0), vT)
\leq t\bigr).
\end{eqnarray*}
According to Lemma~\ref{lem:diameter}, for any $m \in\N$ there
exists a constant $\const{2}_m$ such that for $t \geq\sqrt{T}$, we have
\begin{eqnarray*}
\mathbf{P}\bigl(\diam(\flow_{t+\alpha}(B_R(0)) )< 1\bigr) \leq
\const{2}_mT^{-m}.
\end{eqnarray*}
Thus we get for $t \geq\sqrt{T}$
%
%
\begin{eqnarray} \label{eq:firstEst}\qquad
\mathbf{P}\bigl(\tau^R(B_{\tilde\eps T}(0), vT) \leq t\bigr) \leq
\frac
{1}{p_1} \mathbf{P}\bigl(\tau^{\tilde\eps T}( B_{R}(0), vT) \leq t +
\alpha\bigr) + \frac{\const{2}_m}{p_1}T^{-m}.
\end{eqnarray}
Further we have
%
%
\begin{eqnarray} \label{eq:secondEst}
&&\mathbf{P}\biggl(\tau^{\tilde\eps T}( B_{R}(0), vT) \leq\biggl
(\|v\|^R -
\frac{\eps}{2} \biggr)T\biggr) \notag\hspace*{-20pt}\\[-2pt]
&& \quad \leq\mathbf{P}\biggl(\tau^{\tilde\eps T}( B_{R}(0), vT)
\leq\biggl(\|v\|^R - \frac{\eps}{2} \biggr)T ;  \tau^R(
B_R(0),vT) > \biggl(\|v\|^R - \frac
{\eps}{6} \biggr)T\biggr)\hspace*{-20pt}  \\[-2pt]
&& \qquad{} + \mathbf{P}\biggl(\tau^R( B_R(0),vT) \leq
\biggl(\|v\|^R -
\frac{\eps}{6} \biggr)T\biggr),
\nonumber\hspace*{-20pt}
\end{eqnarray}
where the second term converges to $0$ for $T \to\infty$ by Lemma
\ref{lem:holger2}. To estimate the first term consider an $R$-net on
$\partial B_{\tilde\eps T}(vT)$, that is, there exists $N(\tilde\eps
T) \in\N$ and points $Tw_1,\ldots, Tw_{N(\tilde\eps T)} \in
\partial B_{\tilde\eps T}(0)$ such that
\begin{eqnarray*}
\partial B_{\tilde\eps T}(vT) \subseteq\bigcup_{i = 1}^{N(\tilde
\eps T)}B_R\bigl((v+w_i)T\bigr),
\end{eqnarray*}
where $N(\tilde\eps T)$ grows at most polynomial in $T$ for a fixed
degree $\tilde m \in\N$. Thus we get, estimating the first term in
\rref{eq:secondEst}, using isotropy of the flow,
%
%
\begin{eqnarray} \label{eq:thirdEst}
&&\mathbf{P}\biggl(\tau^{\tilde\eps T}( B_{R}(0), vT) \leq\biggl
(\|v\|^R -
\frac{\eps}{2} \biggr)T ; \tau^R( B_R(0),vT) > \biggl(\|v\|^R -
\frac{\eps}{6} \biggr)T\biggr) \notag\\
&& \qquad \leq\sum_{i = 1}^{N(\tilde\eps T)} \mathbf{P}\biggl
(\tau
^R( B_R(0),vT)
\nonumber\\[-2pt]
&&  \hphantom{\leq\sum_{i = 1}^{N(\tilde\eps T)} \mathbf{P}\biggl
(}\qquad
> \biggl(\|v\|^R - \frac{\eps}{6} \biggr)T
| \tau^{R}\bigl( B_{R}(0), (v+w_i)T\bigr)\leq\biggl(\|v\|^R - \frac{\eps}{2}
\biggr)T\biggr)\\[-2pt]
&& \qquad \leq\sum_{i = 1}^{N(\tilde\eps T)} \mathbf{P}\biggl
(\tau
^R\bigl(\flow_{\tau^R( B_R(0),(v+w_i)T)}( B_R(0)),vT\bigr) > \frac
{\eps}{3}
T\biggr)\notag\\[-2pt]
&& \qquad \leq{N(\tilde\eps T)} \sup_{\gamma\in\mathcal{C}_R}
\mathbf{P}\biggl(\tau^R(\gamma, e_1\tilde\eps T) > \frac{\eps
}{3} T\biggr)
\notag\\[-2pt]
&& \qquad \leq{N(\tilde\eps T)} \sup_{\gamma\in\mathcal{C}_R}
\mathbf{P}\biggl(\tau^R(\gamma, e_1\tilde\eps T) > \biggl(\tilde
\eps\|
e_1\|^R+\frac{\eps}{6} \biggr) T\biggr).
\nonumber
\end{eqnarray}
This last probability converges according to Lemma~\ref{lem:holger2},
uniformly in $\gamma\in\mathcal{C}_R$, as $o(T^{-m})$ for $m >
\tilde m$ to 0 as $T\to\infty$. Hence combining \rref{eq:firstEst},
\rref{eq:secondEst} and \rref{eq:thirdEst}, we get for $t = (\|
v\|^R - \eps)T$ and $T \geq\frac{2\alpha}{\eps}$,
\begin{eqnarray*}
&&\mathbf{P}\bigl(\tau^R(B_{\tilde\eps T}(0), vT) \leq(\|v\|^R -
\eps)T\bigr) \\[-2pt]
&& \qquad \leq\frac{1}{p_1} \mathbf{P}\biggl(\tau^{\tilde\eps T}(
B_{R}(0), vT) \leq\biggl(\|v\|^R - \frac{\eps}{2} \biggr)T\biggr)
+ \frac
{\const{2}_m}{p_1} T^{-m}\\[-2pt]
&& \qquad \to0,
\end{eqnarray*}
as $T \to\infty$, which completes the proof.\vadjust{\goodbreak}
\end{pf}

Using Lemma~\ref{lem:convBepsT} we will show that all time-scaled
trajectories starting in a linearly growing set behave like a Lipschitz
function for a given mesh size $\Delta t$.

%
\begin{lemma} \label{lem:growingBepsT}
Let $\eps\in(0, 1)$ and $\Delta t \in(0,1)$. Then for $0 < \tilde
\eps\leq\frac{K(1+\eps/2)\Delta t \eps}{6(4 + \eps)}$, we have
\begin{eqnarray*}
\lim_{T\to\infty} \mathbf{P}\biggl(\sup_{x \in B_{\tilde\eps
T}(0)} \biggl|
\frac{1}{T}x - \frac{1}{T}\flow_{\Delta t T}(x)\biggr| \geq\Delta
t K (1
+ \eps)\biggr) = 0.
\end{eqnarray*}
\end{lemma}

\begin{pf}
Since $| v| = K\|v\|^R$ and $\tilde\eps\leq\Delta t K \frac{\eps
}{2}$, we have for some constant $c^*$, specified below, and $T$ large,
\begin{eqnarray*}
&&\mathbf{P}\Bigl(\sup_{x \in B_{\tilde\eps T}(0)} | x - \flow
_{\Delta t
T}(x)| \geq\Delta t K (1 + \eps) T\Bigr)\\[-2pt]
&& \qquad \leq\mathbf{P}\biggl(\sup_{x \in B_{\tilde\eps T}(0)} \|
\flow_{\Delta t T}(x)\|^R \geq\Delta t \biggl(1 + \frac{\eps
}{2} \biggr)T\biggr) \\[-2pt]
&& \qquad \leq\mathbf{P}\biggl(\exists x \in B_{\tilde\eps T}(0)
\dvtx
\|\flow_{\Delta t T}(x)\|^R = \Delta t \biggl(1 + \frac{\eps
}{2} \biggr)T\biggr)\\[-2pt]
&& \qquad \quad{} +\mathbf{P}\biggl(\inf_{x \in B_{\tilde\eps
T}(0)} \|\flow
_{\Delta t T}(x)\|^R > \Delta t \biggl(1 + \frac{\eps}{2} \biggr
)T\biggr)\\[-2pt]
&& \qquad \leq\mathbf{P}\biggl(\exists v \in\R^2\dvtx \|v\|^R =
\Delta
t \biggl(1 + \frac{\eps}{2} \biggr) ; \flow_{\Delta t
T}(B_{\tilde\eps T}(0)) \cap B_R (vT ) \neq\varnothing\biggr)\\[-2pt]
&& \qquad \quad{} +\mathbf{P}\Bigl(\inf_{x \in B_{\tilde\eps
T}(0)} | \flow
_{\Delta t T}(x)| > c^* \log(\Delta t T)\Bigr)\\[-2pt]
&& \qquad \leq\mathbf{P}\biggl(\exists v \in\R^2\dvtx \|v\|^R=
\Delta
t \biggl(1 + \frac{\eps}{2} \biggr) ; \tau^R (B_{\tilde\eps
T}(0),vT ) \leq\Delta t T\biggr)\\[-2pt]
&& \qquad \quad{} + \mathbf{P}\bigl(\diam(\flow_{\Delta t
T}(B_{\tilde\eps
T}(0))) < 1\bigr)\\[-2pt]
&& \qquad \quad{} +\mathbf{P}\Bigl(\inf_{x \in B_{\tilde\eps
T}(0)} | \flow
_{\Delta t T}(x)| > c^* \log(\Delta t T)\Bigr).
\end{eqnarray*}
First observe that~\cite{ss02}, Theorem 4.2, yields the existence of a
constant $c^*$ such that the probability that there exists some $x \in
B_{\tilde\eps T}(0)$, which remains in a logarithmic neighborhood of
the origin, that is, $| \flow_s(x)|\leq c^*\log{s}$ for all $s \geq
\Delta t T$, converges to $1$ for $T \to\infty$. Hence the third
probability converges to $0$, and, because of Lemma~\ref
{lem:diameter}, the second probability converges to $0$ as well. Thus
we get
%
%
\begin{eqnarray} \label{eq:firstLipschitz}
&&\lim_{T\to\infty}\mathbf{P}\Bigl(\sup_{x \in B_{\tilde\eps
T}(0)} | x
- \flow_{\Delta t T}(x)| \geq\Delta t K (1 + \eps) T\Bigr)\notag\\
&& \qquad \leq\lim_{T\to\infty}\mathbf{P}\biggl(\exists v \in\R
^2\dvtx
\|v\|^R = \Delta t \biggl(1 + \frac{\eps}{2} \biggr) ; \tau
^R (B_{\tilde\eps T}(0),vT ) \leq\Delta t T\biggr)\\
&& \qquad =\lim_{T\to\infty}\mathbf{P} \biggl(\underbrace
{\exists v \in\Delta t \partial\mathcal{B} \dvtx \tau^R \bigl
(B_{(\fracc
{\tilde\eps}{1+\eps/2}) T}(0),vT \bigr) \leq\frac{\Delta t}{1 +
\fraca{\eps}{2}} T}_{:= S_1(T)} \biggr),\notag
\end{eqnarray}
where $\mathcal{B}$ denotes the unit ball with respect to the stable
norm. Let now $\delta:= \frac{\eps\Delta t}{16\|e_1\|^R}$ and
$v_1,\ldots,v_N$ a $\delta$-net on $\Delta t\partial\mathcal{B}$.
Because of Lemma~\ref{lem:convBepsT} with $\tilde\eta:= \frac
{\tilde\eps}{(1 + \eps/2)} \leq\frac{K}{6}\frac{\Delta t \eps}{4
+ \eps}$, we have
%
%
\begin{eqnarray} \label{eq:firstconv}
\mathbf{P}(S_2(T)) := \mathbf{P}\biggl(\exists j \dvtx \tau^R
(B_{\tilde
\eta T}(0),v_j T ) \leq\frac{\Delta t}{ (1+\fraca{\eps
}{4} )} T\biggr) \to0.
\end{eqnarray}
Because of the isotropy of the flow, we get
%
%
\begin{eqnarray} \label{eq:secondconv}
&&\mathbf{P}(S_2(T)^c|S_1(T))\notag\\
&& \qquad =\mathbf{P}\biggl(\forall j\dvtx \tau^R (B_{\tilde\eta
T}(0),v_j T ) > \frac{\Delta t}{ (1+\fraca{\eps}{4}
)} T \Big| S_1(T)\biggr)\notag\\
&& \qquad \leq\mathbf{P}\biggl(\forall j\dvtx | v-v_j| \leq\delta;
\tau^R (B_{\tilde\eta T}(0),v_j T ) > \frac{\Delta
t}{ (1+\fraca{\eps}{4} )} T \Big| S_1(T)\biggr)\notag\\
&& \qquad \leq\mathbf{P} \biggl(\forall j\dvtx | v-v_j| \leq
\delta
; \tau^R\bigl(\flow_{\tau^R(B_{\tilde\eta T}(0),vT)}(B_{\tilde
\eta
T}(0)), v_j T\bigr)\notag \\
&&\hspace*{62.5pt} \qquad \hphantom{\leq\mathbf{P} \biggl(}
>\biggl( \frac{1}{ (1+\fraca{\eps}{4}
)}-\frac{1}{ (1+\fraca{\eps}{2} )} \biggr)\Delta t T
\Big| S_1(T) \biggr)\\
&& \qquad \leq\sup_{\gamma\in\mathcal{C}_R} \mathbf{P}\biggl
(\tau
^R(\gamma, \delta e_1 T) > \biggl( \frac{1}{ (1+\fraca{\eps
}{4} )}-\frac{1}{ (1+\fraca{\eps}{2} )}
\biggr)\Delta t T\biggr)\notag\\
&& \qquad \leq\sup_{\gamma\in\mathcal{C}_R} \mathbf{P}\biggl
(\tau
^R(\gamma, \delta e_1 T) > \biggl(\delta\|e_1\|^R + \frac{\eps
}{16}\Delta t \biggr) T\biggr),\notag
\end{eqnarray}
which converges to $0$ for $T \to\infty$ according to Lemma \ref
{lem:holger2}. Combining \rref{eq:firstconv} and \rref{eq:secondconv}
now yields
\begin{eqnarray*}
\mathbf{P}(S_1(T)) \leq\mathbf{P}({S_2(T)}^c|S_1(T)) + \mathbf
{P}(S_2(T)) \to0,
\end{eqnarray*}
which completes the proof because of \rref{eq:firstLipschitz}.
\end{pf}

The event that
between two supporting points of the grid (chosen sufficiently close)
the trajectories move not too
quickly will be treated in the following lemma.
It is an application of the chaining techniques
introduced by Scheutzow~\cite{sch09}.

%
\begin{lemma} \label{lem:chaining}
For any bounded $\mathcal{X} \subseteq\R^2$, $a > 0$ and any
partition $0 = t_0 < t_1 <\cdots< t_n = 1$ of $[0,1]$ with $\Delta t
:= \max_{i}| t_{i-1}-t_i| < \frac{a^2}{12\kappa}$ with $\kappa:=
\max\{\beta_L;\beta_N\}$, we have
\begin{eqnarray*}
\lim_{T \to\infty}
\mathbf{P} \biggl(\sup_{x \in\mathcal{X}} \max_{i} \sup_{t_i \leq
t \leq
t_{i+1}}\biggl| \frac{1}{T}\flow_{t_iT}(x)-\frac{1}{T}\flow
_{tT}(x)\biggr| > a \biggr)
= 0.
\end{eqnarray*}
\end{lemma}

\begin{pf}
Denote by $N( \mathcal{X}, \delta)$ the minimal number of closed
balls of radius $\delta> 0$ needed to cover $\mathcal{X}$. Let
$\mathcal{X}_j$, $j = 1, \ldots, N(\mathcal{X},e^{-3\kappa T})$ be
compact sets of diameter at most $e^{-3\kappa T}$, which cover
$\mathcal{X}$, and choose arbitrary points $x_j \in\mathcal{X}_j$.
Then there exists a constant $L > 0$ (depending only on $\mathcal{X}$)
such that
\begin{eqnarray*}
N(\mathcal{X},e^{-3\kappa T}) \leq L e^{3\kappa T}.
\end{eqnarray*}
We have
\begin{eqnarray*}
\mathbf{P}\biggl(\sup_{x \in\mathcal{X}} \max_{i} \sup_{t_i \leq
t \leq
t_{i+1}}\biggl| \frac{1}{T}\flow_{t_iT}(x)-\frac{1}{T}\flow
_{tT}(x)\biggr| > a\biggr)
\leq P_1 + P_2,
\end{eqnarray*}
where
\begin{eqnarray*}
P_1 := L e^{3\kappa T}n \max_{i,j} \mathbf{P}\Bigl(\sup_{t_i \leq t
\leq
t_{i+1}}| \flow_{t_iT}(x_j)-\flow_{tT}(x_j)| > Ta - 1\Bigr)
\end{eqnarray*}
and
\begin{eqnarray*}
P_2 := L e^{3\kappa T} n \max_{j} \mathbf{P}\Bigl(\sup_{0 \leq t
\leq
1}\diam(\flow_{tT}(\mathcal{X}_j)) > 1\Bigr).
\end{eqnarray*}
Because of the temporal and spatial homogeneity of the flow, and since
the one-point motion is Brownian, we get, by denoting a one-dimensional
Brownian motion by $W$,
\begin{eqnarray*}
P_1 &\leq&2Ln e^{3\kappa T} \mathbf{P}\biggl(\sup_{0 \leq s \leq
\Delta t
T}| W_s| > \frac{Ta - 1}{\sqrt{2}}\biggr)\\
&\leq&8Ln e^{3\kappa T}\frac{\sqrt{\Delta t}}{(a-1)\sqrt{2 \pi T}}
\exp\biggl(-\frac{a^2}{4\Delta t}T \biggr) \\
&=& 8 L n \frac{\sqrt{\Delta t}}{(a-1)\sqrt{2 \pi T}} \exp
\biggl( \biggl(3\kappa-\frac{a^2}{4\Delta t} \biggr)T \biggr) \to0
\end{eqnarray*}
for $T \to\infty$; see~\cite{kar91}, Problem II.8.2. On the other
hand we use Theorem 2.1 of~\cite{sch09} (see Theorem \ref
{thm:chaining} in the \hyperref[appm]{Appendix}) to bound $P_2$, which
gives an upper
bound on the exponential decay of the probability of the expansion of
exponentially shrinking sets, that is, the sets $\mathcal{X}_j$.
Because of Lemma~\ref{lem:kappa} the derivative of the quadratic
variation of the difference $M(t,x) - M(t,y)$, where $M$ is the
generating isotropic Brownian field, satisfies the Lipschitz property
with $\kappa= \max\{\beta_L;\beta_N\} > 0$. Lemma 2.6 of \cite
{sch09} ensures that Theorem 2.1 can be applied with $\sigma^2 =
\kappa$ and $\Lambda= \frac{\kappa}{2}$. Hence there exists $\tilde
T$ such that for $T \geq\tilde T$,
\begin{eqnarray*}
P_2 &\leq& L e^{3\kappa T} n \max_{j} \mathbf{P}\Bigl(\sup_{x,y
\in
\mathcal{X}_j} \sup_{0\leq s \leq T} | \flow_s(x) - \flow_s(y)| >
1\Bigr) \\
&\leq& L e^{3\kappa T} n \exp\biggl(- \biggl(\frac{1}{2\kappa}
\biggl(3\kappa-\frac{\kappa}{2} \biggr)^2 + \frac{\kappa}{16}
\biggr)T \biggr)
=L n \exp\biggl(- \frac{\kappa}{16}T \biggr) \to0,
\end{eqnarray*}
for $T \to\infty$, which completes the proof.
\end{pf}

The next lemma shows that it is sufficient to analyze the Lipschitz
behavior of the time-scaled trajectories to get rid of the infimum over
all Lipschitz functions.

%
\begin{lemma}\label{lem:triineq}
For any $\eps> 0$, $\mathcal{X} \subseteq\R^2$ and any partition $0
= t_0 < t_1 <\cdots< t_n = 1$ of $[0,1]$, we have
\begin{eqnarray*}
\biggl\{\sup_{x \in\mathcal{X}} \inf_{f \in\Lip(K)} \max_{i}
\biggl|
\frac{1}{T}\flow_{t_{i} T}(x) - f(t_i)\biggr| > \frac{\eps}{3}
 \biggr\}
\subseteq S_1 \cup S_2,
\end{eqnarray*}
where
\begin{eqnarray*}
S_1 &:= & \biggl\{\sup_{x \in\mathcal{X}} \max_i \frac
{1}{(t_{i+1}-t_i)} \biggl| \frac{1}{T}\flow_{t_iT}(x) - \frac
{1}{T}\flow
_{t_{i+1}T}(x)\biggr| > \biggl(K + \frac{\eps}{3} \biggr) \biggr\}
\end{eqnarray*}
and
\begin{eqnarray*}
S_2 &:= & \biggl\{\sup_{x \in\mathcal{X}} \max_i \biggl| \frac
{1}{t_iT}\flow_{t_iT}(x)\biggr| > \biggl(K + \frac{\eps}{3} \biggr)
\biggr\}.
\end{eqnarray*}
\end{lemma}

\begin{pf}
Let $x \in\mathcal{X}$. Then
\begin{eqnarray*}
\max_i \frac{1}{(t_{i+1}-t_i)} \biggl| \frac{1}{T}\flow_{t_iT}(x)
- \frac
{1}{T}\flow_{t_{i+1}T}(x)\biggr| \leq\biggl(K + \frac{\eps}{3}
\biggr)
\end{eqnarray*}
and
%
%
\begin{eqnarray}\label{eq:Totalboundonflow}
\max_i \biggl| \frac{1}{t_iT}\flow_{t_iT}(x)\biggr| \leq\biggl(K
+ \frac{\eps
}{3} \biggr)
\end{eqnarray}
imply that the function $f_x$, defined by
\begin{eqnarray*}
f_x(0) = 0 \quad\mbox{and}\quad f_x(t_i) := \frac{1}{T}\flow
_{t_iT}(x)\frac{K}{ (K + \fraca{\eps}{3} )},
\qquad i \in\{
1, \ldots, n\}
\end{eqnarray*}
and linear interpolation for $t \in(t_i, t_{i+1})$, is Lipschitz
continuous with Lipschitz constant $K$, hence $f_x\in\Lip(K)$.
Further, by \rref{eq:Totalboundonflow} and definition of $f_x$, we have
\begin{eqnarray*}
\max_i \biggl| \frac{1}{T}\flow_{t_iT}(x) - f_x(t_i)\biggr| \leq
\frac{\eps}{3},
\end{eqnarray*}
which completes the proof by taking complements and unifying over all
$x \in\mathcal{X}$.\vadjust{\goodbreak}
\end{pf}

Finally we provide the proof of Theorem~\ref{thm:upperbound}.

\begin{pf*}{Proof of Theorem~\ref{thm:upperbound}}
For any partition $0 = t_0 < t_1 <\cdots< t_n = 1$ of $[0,1]$ with
\begin{eqnarray*}
\Delta t := \max_i\{t_{i+1} - t_i\} \leq\min\biggl\{\frac{\eps
}{3(K + \fraca{\eps}{3})}; \frac{\eps^2}{108\kappa} \biggr\},
\end{eqnarray*}
by the triangle inequality and according to Lemma~\ref{lem:triineq},
we have
\begin{eqnarray*}
\mathbf{P}\Bigl(\sup_{g \in F_T(\mathcal{X})}d (g, \Lip(K) )
> \eps\Bigr) &=& \mathbf{P}\biggl(\sup_{x \in\mathcal{X}} \inf
_{f \in\Lip
(K)} \biggl\|\frac{1}{T}\flow_{\cdot T}(x) - f\biggr\|_\infty>
\eps\biggr) \\&\leq&
P_1 + P_2 + P_3,
\end{eqnarray*}
where
\begin{eqnarray*}
P_1 &:=& \mathbf{P}\biggl(\sup_{x \in\mathcal{X}} \max_i \frac
{1}{(t_{i+1} - t_i)} \biggl| \frac{1}{T}\flow_{t_iT}(x) - \frac
{1}{T}\flow
_{t_{i+1}T}(x)\biggr| > K \biggl(1 + \frac{\eps}{3} \biggr)\biggr)
\end{eqnarray*}
and
\begin{eqnarray*}
P_2 &:=& \mathbf{P}\biggl(\sup_{x \in\mathcal{X}} \max_i \biggl|
\frac
{1}{t_iT}\flow_{t_iT}(x)\biggr| > K \biggl(1 + \frac{\eps}{3}
\biggr)\biggr)
\end{eqnarray*}
and
\begin{eqnarray*}
P_3 &:=& \mathbf{P}\biggl(\sup_{x \in\mathcal{X}} \max_i \sup
_{t_i \leq
t \leq t_{i+1}} \biggl| \frac{1}{T}\flow_{t_iT}(x) - \frac
{1}{T}\flow
_{tT}(x)\biggr| > \frac{\eps}{3}\biggr).
\end{eqnarray*}
According to Lemma~\ref{lem:chaining}, since $\Delta t \leq\frac
{\eps^2}{108\kappa}$, we immediately get $P_3 \to0$. According to
\rref{cor:upperbound} we have
\begin{eqnarray*}
P_2 \leq\sum_{i = 1}^n \mathbf{P}\biggl(\flow_{t_iT}(\mathcal{X})
\nsubseteq t_i T \biggl(1 + \frac{\eps}{3} \biggr) \mathcal
{B}\biggr) \to0,
\end{eqnarray*}
where $\mathcal{B}$ denotes the unit ball with respect to the stable
norm. For the convergence of $P_1$ it hence suffices to show that for
all $i \in\{1, \ldots, n\}$,
\begin{eqnarray*}
&&\mathbf{P} \biggl(\sup_{x \in\mathcal{X}} \biggl| \frac
{1}{T}\flow
_{t_iT}(x) - \frac{1}{T}\flow_{t_{i+1}T}(x)\biggr| \\
&&\hphantom{\mathbf{P}\biggl(} > (t_{i+1} - t_i) K
\biggl(1 + \frac{\eps}{3} \biggr) \Big| \flow_{t_iT}(\mathcal
{X})\subseteq t_i
T(1+\eps
)\mathcal{B}\biggr)
\end{eqnarray*}
converges to $0$ for $T \to\infty$. Let $\tilde\eps\leq\frac
{K(1+\eps/6)\tilde\Delta t \eps}{18(4 + \eps/3)}$, where $\tilde
\Delta t := \min_i\{t_{i+1} - t_i\}$; then there exists for fixed $i
\in\{1, \ldots, n\}$ an integer $N \in\N$ and $v_1, \ldots, v_N
\in t_i (1 + \eps) \mathcal{B}$ such that
\begin{eqnarray*}
t_i T(1+\eps) \mathcal{B} \subseteq\bigcup_{j = 1}^N B_{\tilde\eps
T}(v_j T).
\end{eqnarray*}
Hence we get, using isotropy of the flow,
\begin{eqnarray*}
&&\mathbf{P} \biggl(\sup_{x \in\mathcal{X}} \biggl| \frac
{1}{T}\flow
_{t_iT}(x) - \frac{1}{T}\flow_{t_{i+1}T}(x)\biggr| \\
&& \hphantom{\mathbf{P} \biggl(} > (t_{i+1} - t_i) K
\biggl(1 + \frac{\eps}{3} \biggr) \Big| \flow_{t_iT}(\mathcal
{X})\subseteq t_i T(1 + \eps
)\mathcal{B} \biggr)\\
&& \qquad \leq N \cdot\mathbf{P}\biggl(\sup_{x \in B_{\tilde\eps
T}(0)} \biggl| \frac{1}{T}x - \frac{1}{T}\flow
_{(t_{i+1}-t_i)T}(x)\biggr| >
(t_{i+1} - t_i) K \biggl(1 + \frac{\eps}{3} \biggr) \biggr)\\
&& \qquad \to0
\end{eqnarray*}
for $T \to\infty$, according to Lemma~\ref{lem:growingBepsT}. Thus
the assertion is proved.
\end{pf*}

\subsection{Lower bound} \label{sec:lowerBound}

This section is devoted to the proof of the lower bound of Theorem \ref
{thm:mainthm}, that is, the following theorem.

%
\begin{theorem} \label{thm:lowerBound}
For any $\eps> 0$ and $\mathcal{X} \in\mathcal{C}_R$, we have
\begin{eqnarray*}
\lim_{T \to\infty} \mathbf{P}\Bigl(\sup_{f \in\Lip(K)} d(f,
F_T(\mathcal{X})) > \eps\Bigr) = 0,
\end{eqnarray*}
where $K$ is the Euclidean radius of the stable norm unit ball; see
Section~\ref{sec:stablenorm}.
\end{theorem}

The proof of Theorem~\ref{thm:lowerBound} is divided into several
steps. Since the Lipschitz functions are compact with respect to the
supremum norm, the problem can be reduced to a finite set of Lipschitz
functions; see the proof of Theorem~\ref{thm:lowerBound}. The main
idea is then to show that for any given Lipschitz function, there
exists a point in the initial set such that the image of this point,
under the action of the flow, approximates the Lipschitz function on
a discrete grid (Lemma~\ref{lem:discreteEst}). Further, Lemma \ref
{lem:chaining} shows that between two supporting points, if chosen
sufficiently close, the trajectories move not too quickly.

%
\begin{lemma} \label{lem:discreteEst}
For any $\eps> 0$, $f \in\Lip(K-\eps)$, $\mathcal{X} \in\mathcal
{C}_R$ and any partition $0 = t_0 < t_1 <\cdots< t_n = 1$ of $[0,1]$,
we have
\begin{eqnarray*}
\lim_{T \to\infty} \mathbf{P}\biggl(\inf_{x \in\mathcal{X}}
\max_{i}
\biggl| \frac{1}{T}\flow_{t_iT}(x) - f(t_i)\biggr| \leq\eps\biggr
) = 1.
\end{eqnarray*}
\end{lemma}

\begin{pf}
Consider the following sequence of random subsets of $\R^2$:
\begin{eqnarray*}
\mathcal{X}_0^{(T)} &:=& \mathcal{X},\\
\mathcal{X}_i^{(T)} &:=& \flow_{t_{i-1}T,t_iT} \bigl(\mathcal
{X}_{i-1}^{(T)} \bigr) \cap B_{T^{2/3}}(Tf(t_i))
\end{eqnarray*}
for $i=1,\ldots,n$, which is the part of $\flow_{t_iT}(\mathcal{X})$
that has been close (in linear scaling) to $Tf(t_j)$ for all $0 \leq j
\leq i$. Further define\vadjust{\goodbreak} the set [abbreviating $\tau^R (\mathcal
{X}_{i-1}^{(T)}, Tf(t_i),Tt_{i-1} )$ by $\tau^R_i$]
\begin{eqnarray*}
\gamma_i^{(T)} &:=& \flow_{t_{i-1}T,t_{i-1}T+\tau^R_i} \bigl
(\mathcal
{X}_{i-1}^{(T)} \bigr) \cap B_{2R}(Tf(t_i))
\end{eqnarray*}
for $i=1,\ldots,n$, which is the part of $\mathcal{X}_{i-1}^{(T)}$
that is at first in a $2R$-neighborhood of $Tf(t_i)$. Observe that
$\mathcal{X}_{i-1}^{(T)} \neq\varnothing$ implies that $\tau^R_i$ is
almost surely finite. To simplify notation we will denote the largest
(with respect to the diameter) connected component of $\mathcal
{X}_i^{(T)}$ and $\gamma_i^{(T)}$, respectively, by the same symbol.
Let $A_i^{(T)}$ be the event that $\mathcal{X}_{i-1}^{(T)}$ reaches an
$R$-neighborhood of $Tf(t_i)$ in time, that is,
\begin{eqnarray*}
A_i^{(T)} &:= & \bigl\{\tau^R \bigl(\mathcal{X}_{i-1}^{(T)},
Tf(t_i),Tt_{i-1} \bigr) \leq(t_i - t_{i-1})T \bigr\}
\end{eqnarray*}
for $i=1,\ldots,n$, and $B_i^{(T)}$, the event that there exists a
point in the first intersection of $\mathcal{X}_{i-1}^{(T)}$ with an
$R$-neighborhood of $Tf(t_i)$ that stays close (in linear scaling) to
$Tf(t_i)$ up to time $t_iT$, and $\mathcal{X}_{i-1}^{(T)}$ is large at
time $t_i$, that is, on
\begin{eqnarray*}
\bigl\{\tau^R \bigl(\mathcal{X}_{i-1}^{(T)}, Tf(t_i),Tt_{i-1}
\bigr) \leq(t_i - t_{i-1})T \bigr\},
\end{eqnarray*}
that is [abbreviating $\tau^R (\mathcal{X}_{i-1}^{(T)},
Tf(t_i),Tt_{i-1} )$ by $\tau^R_i$],
\begin{eqnarray*}
B_i^{(T)} &:=& \Bigl\{\inf_{x \in\gamma_{i-1}^{(T)}}\sup_{t_{i-1}T
+ \tau^R_i \leq t \leq t_iT}| \flow_{t_{i-1}T + \tau^R_i,t}(x) -
Tf(t_i)| \leq T^{2/3} ; \\
&&\hspace*{120pt}\hphantom{\Bigl\{} \diam\bigl(\flow_{t_{i-1}T,
t_iT}\bigl(\mathcal
{X}_{i-1}^{(T)} \bigr) \bigr) \geq1 \Bigr\}.
\end{eqnarray*}
Hence we get by construction that if there exists $x \in\mathcal{X}$
such that $\flow_{\cdot}(x)$ reaches successively the
$R$-neighborhoods of $Tf(t_i)$ for all  $i\in\{1, \ldots, n\}$ in time
(before time $t_iT$) and is still close to these points at time $t_iT$,
then the time-scaled trajectory $\frac{1}{T}\flow_{\cdot T}(x)$
starting in this particular $x$ is close to the Lipschitz function $f$
at the time $t_i$ for all $i \in\{0,\ldots, n\}$, that is,
%
%
\begin{eqnarray} \label{eq:estimateAB}
&&\mathbf{P}\biggl(\inf_{x \in\mathcal{X}} \max_{i} \biggl|
\frac{1}{T}\flow
_{t_iT}(x) - f(t_i)\biggr| \leq\eps\biggr)\nonumber\\
 && \qquad \geq\mathbf{P}\Biggl
(\bigcap_{i=1}^n
A_i^{(T)} \cap\bigcap_{i=1}^n B_i^{(T)}\Biggr)\\
&& \qquad = \mathbf{P}\bigl(A_1^{(T)}\bigr)\mathbf{P}\bigl
(B_1^{(T)} |
A_1^{(T)}\bigr)\cdots\mathbf{P}\Biggl(B_n^{(T)} \Big| \bigcap_{i=1}^n
A_i^{(T)} \cap\bigcap_{i=1}^{n-1} B_i^{(T)}\Biggr).
\nonumber
\end{eqnarray}
Observe that the conditional distribution $\mathcal{L} (\tau
^R (\mathcal{X}_{i-1}^{(T)}, Tf(t_i),Tt_{i-1} )
|\mathcal{X}_{i-1}^{(T)} )$ coincides with the conditional
distribution $\mathcal{L} (\tau^R (\mathcal
{X}_{i-1}^{(T)}, Tf(t_i) ) |\mathcal{X}_{i-1}^{(T)} )$
for $i \in\{1,\ldots, n\}$, and hence the results from
Section \ref
{sec:stablenorm} are applicable.

For any $k \in\{1,\ldots, n\}$, because of the Markov property \rref
{eq:markovprop} of the flow, we have
%
%
\begin{eqnarray} \label{eq:estimateOnA1}
&&\mathbf{P}\Biggl(A_k^{(T)}\Big|\bigcap_{i=1}^{k-1} A_i^{(T)} \cap
\bigcap_{i=1}^{k-1} B_i^{(T)}\Biggr)\notag\\
&& \qquad = \mathbf{P}\Biggl(\tau^R \bigl(\mathcal
{X}_{k-1}^{(T)},Tf(t_k),Tt_{k-1} \bigr) \leq(t_k - t_{k-1})T
\Big| \bigcap_{i=1}^{k-1} A_i^{(T)} \cap\bigcap_{i=1}^{k-1}
B_i^{(T)}\Biggr)
\nonumber
\\[-8pt]
\\[-8pt]
&& \qquad \geq\inf_{\gamma\in\mathcal{C}_R} \inf_{v \in
B_{1}(0)} \mathbf{P}\bigl(\tau^R \bigl(\gamma,T\bigl(f(t_k) -
f(t_{k-1})\bigr) +
vT^{2/3} \bigr) \leq(t_k - t_{k-1})T \bigr)\notag\\
&& \qquad \geq1 - \sup_{\gamma\in\mathcal{C}_R} \mathbf
{P}\biggl(\tau^R \bigl(\gamma,T\bigl(f(t_k) - f(t_{k-1})\bigr)
\bigr) > (t_k -
t_{k-1})\frac{T}{1+\eps/K}\biggr)\notag\\
&& \qquad \quad{} -\sup_{\gamma\in\mathcal{C}_R} \sup_{v \in
B_{1}(0)} \mathbf{P}\biggl(\tau^R(\gamma,vT^{2/3}) > (t_k -
t_{k-1})\frac
{\eps}{1+\eps/K}T\biggr).
\nonumber
\end{eqnarray}
Because of the isotropy of the flow, the last probability reduces to
\begin{eqnarray*}
\sup_{\gamma\in\mathcal{C}_R} \mathbf{P}\biggl(\tau^R(\gamma
,e_1T^{2/3}) > (t_k - t_{k-1})\frac{\eps}{1+\eps/K}T\biggr) \to0,
\end{eqnarray*}
and converges to $0$ according to Lemma~\ref{lem:holger2}. Since $f
\in\Lip(K-\eps)$ and $| v| = K \|v\|^R$, we have $\|f(t_k) -
f(t_{k-1})\|^R \leq(t_k - t_{k-1}) (1-\frac{\eps}{K} )$,
which\ implies, because of Lemma~\ref{lem:holger2},
\begin{eqnarray*}
&&\hspace*{-2pt}\sup_{\gamma\in\mathcal{C}_R} \mathbf{P}\biggl(\tau^R \bigl
(\gamma
,T\bigl(f(t_k) - f(t_{k-1})\bigr) \bigr) > (t_k - t_{k-1})\frac
{T}{1+\eps/K}\biggr)\\
&& \hspace*{-2pt}\qquad \leq\sup_{\gamma\in\mathcal{C}_R} \mathbf{P}\biggl
(\tau
^R \bigl(\gamma,T\bigl(f(t_k) - f(t_{k-1})\bigr) \bigr)  > \|f(t_k) - f(t_{k-1})\|^R\frac{1}{1 - (\eps
/K)^2}T \biggr)\notag\\
&& \hspace*{-2pt}\qquad \to0,\notag
\end{eqnarray*}
and hence convergence to $0$ of the first probability in \rref
{eq:estimateOnA1}.
On the other hand, we get for $k \in\{1,\ldots, n\}$, by fixing some
$\tilde x_{k-1} \in\gamma_{k-1}^{(T)}$ for $T$ large [abbreviating
$\tau^R (\mathcal{X}_{k-1}^{(T)}, Tf(t_k),Tt_{k-1} )$ by
$\tau^R_k$], 
%
%
\begin{eqnarray} \label{eq:estimateOnB}
&&\mathbf{P}\Biggl(B_k^{(T)}\Big|\bigcap_{i=1}^{k} A_i^{(T)} \cap
\bigcap
_{i=1}^{k-1} B_i^{(T)}\Biggr)\notag\\
&& \qquad \geq\mathbf{P} \Biggl(\sup_{t_{k-1}T + \tau^R_k \leq t
\leq t_kT} | \flow_{t_{k-1}T + \tau^R_k,t}(\tilde x_{k-1}) - Tf(t_k)|
\nonumber
\\[-8pt]
\\[-8pt]
&& \hspace*{73pt}\qquad\hphantom{\geq\mathbf{P} \Biggl(}\leq
T^{2/3} \Big|\bigcap_{i=1}^{k} A_i^{(T)} \cap\bigcap_{i=1}^{k-1}
B_i^{(T)} \Biggr)\notag\\
&& \qquad \quad{} + \mathbf{P}\Biggl(\diam\bigl(\flow_{t_{k-1}T,t_kT}
\bigl(\mathcal{X}_{k-1}^{(T)} \bigr) \bigr) \geq1 \Big| \bigcap
_{i=1}^{k} A_i^{(T)} \cap\bigcap_{i=1}^{k-1} B_i^{(T)}\Biggr) - 1.
\nonumber
\end{eqnarray}
Since the one-point motions are Brownian, the first term can be
estimated for some $\delta\in(0,1)$ via [denoting by $W =
(W^{(1)},W^{(2)} )$ a $2$-dimensional Brownian motion]
%
%
\begin{eqnarray} \label{eq:brownianEst}
&&\mathbf{P} \Biggl(\sup_{t_{k-1} + \tau^R_k \leq t \leq t_k} |
\flow
_{t_{k-1}T + \tau^R_k,tT}(\tilde x_{k-1}) - Tf(t_k)|  \leq T^{2/3}
\Big| \bigcap_{i=1}^{k} A_i^{(T)} \cap\bigcap
_{i=1}^{k-1} B_i^{(T)} \Biggr)\nonumber
\\
&& \qquad \geq\mathbf{P}\Bigl(\sup_{0 \leq t \leq t_k-t_{k-1}} |
W_{tT}| \leq(1-\delta)T^{2/3}\Bigr)
\\
&& \qquad \geq1 - 8\cdot\mathbf{P}\biggl(W^{(1)}_1 > \frac
{(1-\delta
)}{\sqrt{2(t_k - t_{k-1})}}T^{1/6}\biggr)\nonumber\\
&& \qquad  \to1;
\nonumber
\end{eqnarray}
see~\cite{kar91}, Problem II.8.2. Further, we have, because of Lemma
\ref{lem:diameter},
\begin{eqnarray*}
&&\mathbf{P}\Biggl(\diam\bigl(\flow_{t_{k-1}T,t_kT}\bigl(\mathcal
{X}_{k-1}^{(T)} \bigr) \bigr) \geq1 \Big| \bigcap_{i=1}^{k}
A_i^{(T)} \cap\bigcap_{i=1}^{k-1} B_i^{(T)}\Biggr)\\
&& \qquad \geq\inf_{\gamma\in\mathcal{C}_R} \mathbf{P}\bigl
(\diam
\bigl(\flow_{0,(t_k-t_{k-1})T}(\gamma) \bigr)\geq1\bigr)\\
&& \qquad \to1.
\end{eqnarray*}
This, together with \rref{eq:brownianEst}, yields convergence of \rref
{eq:estimateOnB} to $1$. Combining \rref{eq:estimateOnA1} and \rref
{eq:estimateOnB} via \rref{eq:estimateAB} implies the assertion.
\end{pf}

Finally we provide the proof of Theorem~\ref{thm:lowerBound}.

\begin{pf*}{Proof of Theorem~\ref{thm:lowerBound}}
Because of compactness of the Lipschitz functions with respect to the
supremum norm, we can reduce the problem to a finite set of Lipschitz
functions as follows. Since $\Lip(K-\frac{\eps}{4})$ is compact with
respect to $\|\cdot\|_\infty$ there exists some $N \in\N$ and
$f_1,\ldots,f_N \in\Lip(K-\frac{\eps}{4})$ such that for any $g
\in\Lip(K-\frac{\eps}{4})$, there exists $j \in\{1, \ldots, N\}$ with
\begin{eqnarray*}
\|g - f_j\|_\infty\leq\frac{\eps}{4}.
\end{eqnarray*}
If $f \in\Lip(K)$ then $\frac{K - \fraca{\eps}{4}}{K}f \in\Lip(K
- \frac{\eps}{4})$, and hence for any $f \in\Lip(K)$, because of $\|
f\|_\infty\leq K$, there exists $j \in\{1, \ldots, N\}$ such that
\begin{eqnarray*}
\|f - f_j\|_\infty\leq\biggl\|f - \frac{K - \fraca{\eps
}{4}}{K}f\biggr\|
_\infty+ \biggl\|\frac{K - \fraca{\eps}{4}}{K}f - f_j\biggr\|
_\infty\leq\frac
{\eps}{2}.
\end{eqnarray*}
Thus we get
%
%
\begin{eqnarray} \label{eq:finiteLipEst}
&&\mathbf{P}\biggl(\sup_{f \in\Lip(K)} \inf_{x \in\mathcal{X}}
\biggl\|\frac
{1}{T}\flow_{0,\cdot T}(x) - f\biggr\|_\infty> \eps\biggr) \notag
\\
&& \qquad =\mathbf{P}\biggl(\max_j \mathop{\mathop{\sup}_{f \in
\Lip(K)}}_{|
f - f_j|\leq\fraca{\eps}{2}} \inf_{x \in\mathcal{X}} \biggl\|
\frac
{1}{T}\flow_{0,\cdot T}(x) - f\biggr\|_\infty> \eps\biggr) \\
&& \qquad \leq\sum_{j=1}^N \mathbf{P}\biggl(\inf_{x \in\mathcal{X}}
\biggl\|\frac{1}{T}\flow_{0,\cdot T}(x) - f_j\biggr\|_\infty>
\frac{\eps
}{2}\biggr).\notag
\end{eqnarray}
Now choose a partition $0 = t_0 < t_1 <\cdots< t_n = 1$ of $[0,1]$
with $\Delta t :=\break \max_i\{t_{i+1}-t_i\}\leq\min\{\frac{\eps
^2}{768 \kappa} ; \frac{\eps}{8 K} \}$, where $\kappa:=
\max\{\beta_L;\beta_N\}$. Using the triangle inequality we get for
any $f \in\Lip(K - \frac{\eps}{4})$, since
\begin{eqnarray*}
\max_i \sup_{t_i \leq t \leq t_{i+1}}| f(t_i) - f(t)| \leq\biggl(K -
\frac{\eps}{4} \biggr) \Delta t \leq\frac{\eps}{8},
\end{eqnarray*}
the estimate
%
%
\begin{eqnarray} \label{eq:lowerboundEst}
&&\mathbf{P}\biggl(\inf_{x \in\mathcal{X}} \biggl\|\frac
{1}{T}\flow_{0,\cdot
T}(x) - f_i\biggr\|_\infty> \frac{\eps}{2}\biggr) \notag\\
&& \qquad \leq\mathbf{P}\biggl(\inf_{x \in\mathcal{X}} \max_i
\biggl|
\frac{1}{T}\flow_{0,t_iT}(x) - f(t_i)\biggr| > \frac{\eps
}{4}\biggr) \\
&& \qquad \quad{} +\mathbf{P}\biggl(\sup_{x \in\mathcal{X}}\max
_i \sup
_{t_i \leq t \leq t_{i+1}} \biggl| \frac{1}{T}\flow_{0,t_iT}(x) -
\frac
{1}{T}\flow_{0,tT}(x)\biggr| > \frac{\eps}{8}\biggr).\notag
\end{eqnarray}
Because of Lemma~\ref{lem:discreteEst} the first term in \rref
{eq:lowerboundEst} converges to $0$ for $T \to\infty$, and since
$\Delta t \leq\frac{\eps^2}{768 \kappa}$, Lemma~\ref{lem:chaining}
yields convergence of the second term to $0$. Hence combining \rref
{eq:finiteLipEst} and \rref{eq:lowerboundEst} proves the
assertion.\vspace*{-6pt}
\end{pf*}

\subsection{\texorpdfstring{Proof of Theorem \protect\ref{thm:mainthm}}{Proof of Theorem 3.1}}\label{sec:proof}

\begin{pf*}{Proof of Theorem~\ref{thm:mainthm}}
By definition of the Hausdorff distance, it is sufficient to show
%
%
\begin{equation}\label{eq:mainthmu}
\lim_{T \to\infty} \mathbf{P}\Bigl(\sup_{g \in F_T(\mathcal{X})}
d (g, \Lip( K) ) > \eps\Bigr) = 0
\end{equation}
and
%
\begin{equation}
\label{eq:mainthml}
\lim_{T \to\infty} \mathbf{P}\Bigl(\sup_{f \in\Lip( K)}
d(f,F_T(\mathcal{X})) > \eps\Bigr) = 0.
\end{equation}
For $\mathcal{X} \in\mathcal{C}_R$ equation \rref{eq:mainthmu} is
proved in Section~\ref{sec:upperBound}, namely Theorem \ref
{thm:upperbound}, whereas \rref{eq:mainthml} is proved in Section \ref
{sec:lowerBound}, namely Theorem~\ref{thm:lowerBound}.
For any nontrivial compact connected $\mathcal{X} \subseteq\R^2$ we
need to construct a scaled flow on a diffusively scaled space, such
that the diameter of $\mathcal{X}$ becomes large, and the results of
Theorem~\ref{thm:upperbound} and Theorem~\ref{thm:lowerBound} are applicable.\vadjust{\goodbreak}

Let $r := \diam(\mathcal{X}) > 0$. Define the scaled space $\tilde\R
^2 := \{\frac{x}{r}\dvtx x \in\R^2\}$ equipped with the usual Euclidean
metric, and consider the function
\begin{eqnarray*}
\tilde\flow\dvtx \R_+ \times\R_+ \times\tilde\R^2 \times
\Omega\to
\tilde\R^2; \qquad\tilde\flow_{s,t}(\tilde x,\omega) := \frac
{1}{r} \flow_{r^2s,r^2t}(r \tilde x,\omega).
\end{eqnarray*}
Since $\flow$ is an IBF on $\R^2$, we have that $\tilde\flow$ is
also an IBF on $\tilde\R^2$ with generating isotropic Brownian field
$\tilde M(t,\tilde x) = \frac{1}{r}M(r^2t,r\tilde x)$ for $t \geq0$,
$\tilde x \in\tilde\R^2$ and covariance tensor $\tilde b(\tilde x) =
b(r\tilde x)$ for $\tilde x \in\tilde\R^2$, and thus it has the same
properties as $\flow$, in particular, the top-Lyapunov exponent of
$\tilde\flow$ is strictly positive. By construction of $\tilde\R^2$
the initial set $\frac{1}{r}\mathcal{X}$ has diameter $1$, seen as a
subset of $\tilde\R^2$. Denote the time-scaled trajectories of
$\tilde\flow$ by
\[
\tilde F_T(\mathcal{X},\omega) :=\bigcup_{\tilde x \in ({1}/{r})\mathcal{X}} \biggl\{[0,1] \ni t \mapsto
\frac{1}{T}\tilde \flow_{0,tT}(\tilde x,\omega) \biggr\}.
\]
%
One can easily deduce from \rref{thm:holger1}, using the definition of
$\tilde\flow$, that the Euclidean radius of the unit ball of the
stable norm defined via $\tilde\flow$ in $\tilde\R^2$ is $\tilde K
= r K$. Thus it follows from \rref{eq:mainthmu} and \rref
{eq:mainthml} applied to $\tilde\flow$ that
\begin{eqnarray*}
\lim_{T \to\infty} \mathbf{P}\bigl(d_H(\tilde F_T(\mathcal{X}),
\Lip
(\tilde K)) > \eps\bigr) = 0.
\end{eqnarray*}
By definition of $\tilde F_T(\mathcal{X})$ one sees
that this convergence also holds for the set $\tilde
F_{T/r^2}(\mathcal{X})$
, by definition of $\tilde\flow$,
\[
 F_T(\mathcal{X}) = \frac{1}{r} \tilde F_{T/r^2}(\mathcal{X}) \to \frac{1}{r} \Lip(\tilde  K) = \Lip(K),
\]
%
where convergence is meant in the Hausdorff distance in probability.
This proves the assertion for any nontrivial compact connected set
$\mathcal{X}\subseteq\R^2$.
\end{pf*}

%
\begin{appendix}\label{appm}
\section{An estimate on the covariance function}

One of the general assumptions for stochastic flows is a Lipschitz
property of the derivative of the quadratic variation of the difference
$M(t,x) - M(t,y)$, where $M$ denotes the generating martingale field of
the flow. In case of IBFs this property is achieved by an estimate of
the second derivative of the covariance functions. The following proof
is due to Scheutzow.

%
\begin{lemma} \label{lem:kappa}
Let $\flow$ be an IBF with generating isotropic Brownian field $M$.
The function $\mathcal{A}(t,x,y) := \frac{\dx}{\dx t}\langle M(\cdot
,x) -M(\cdot,y)\rangle_t$ satisfies for all $t \geq0$, $x,y \in\R
^2$, the inequality
\begin{eqnarray*}
\|\mathcal{A}(t,x,y)\| \leq\max\{\beta_L;\beta_N\}| x-y|^2,
\end{eqnarray*}
where $\beta_L$ and $\beta_N$ are as in Section~\ref{sec:IBF}, and
$\|\cdot\|$ denotes the spectral norm on $\R^{2\times2}$.
\end{lemma}

\begin{pf}
Observe that, by definition of the covariance tensor, we have
\begin{eqnarray*}
\mathcal{A}(t,x,y) = 2\bigl(b(0) - b(x-y)\bigr).
\end{eqnarray*}
According to~\cite{vB10}, Lemma 1.6, $x$ is an eigenvector of $b(x)$
to the eigenvalue $B_L(| x|)$, and any vector $x^\bot\neq0$
perpendicular to $x$ is an eigenvector of $b(x)$ to the eigenvalue
$B_N(| x|)$. Since the matrix $\mathcal{A}(t,x,y)$ is symmetric, we have
%
%
\begin{eqnarray} \label{eq:normOfA}
\|\mathcal{A}(t,x,y)\| &=& \bigl\|2\bigl(b(0) - b(x-y)\bigr)\bigr\|\nonumber
\\[-8pt]
\\[-8pt]
&=& 2 \max\{1 - B_L(| x-y|);1 - B_N(| x-y|)\}.
\nonumber
\end{eqnarray}
Now consider an $\R^2$-valued centered Gaussian process $U(x)$, $x \in
\R^2$, with covariances $\mathbf{E}[U_i(x)U_j(y)] = b_{ij}(x-y)$ for
$i,j \in
\{1,2\}$. Then by stationarity and Schwartz's inequality, we have for
$r \geq0$,
\begin{eqnarray*}
B''(r)
&=& \lim_{h \to0} \lim_{\delta\to0} \mathbf{E}\biggl[\frac{U_1(he_1) -
U_1(0)}{h}\frac{U_1(-(r + \delta)e_1) - U_1(-re_1)}{\delta}\biggr]\\
&=& -\mathbf{E}[U_1'(re_1)U_1'(0)] \geq-\mathbf{E}[U_1'(0)^2] = B_L''(0).
\end{eqnarray*}
By Taylor's expansion (\cite{bax86}, Section 2), for each $r > 0$,
there exists some $\theta\in(0,r)$ such that
\begin{eqnarray*}
B_L(r) = B_L(0) + \frac{1}{2}B_L''(\theta)r^2 \geq1 + \frac
{1}{2}B_L''(0)r^2 = 1 - \frac{\beta_L}{2}r^2.
\end{eqnarray*}
The estimate on $B_N$ follows in the same way, so from \rref
{eq:normOfA} we get
\begin{eqnarray*}
\|\mathcal{A}(t,x,y)\| \leq\max\{\beta_L;\beta_N\}| x-y|^2.
\end{eqnarray*}
\upqed
\end{pf}

\section{Chaining at work}

The following theorem is basically Theorem 2.1 of~\cite{sch09}. It
provides an upper bound for the probability that the image of a ball,
which is exponentially small in $T$, attains a fixed diameter up to
time $T$.

%
\begin{theorem} \label{thm:chaining}
Suppose there exist $\Lambda\geq0, \sigma> 0$ such that for each $x,
y \in\R^d$, there exists a standard Brownian motion $W$ such that
\begin{eqnarray*}
| \flow_t(x) - \flow_t(y)| \leq| x-y| \exp(\Lambda t + \sigma
W^*_t), \qquad t \geq0
\end{eqnarray*}
where $W^*_t := \sup_{0 \leq s \leq t} W_s$. Define for $\gamma> 0$
\begin{eqnarray*}
I(\gamma) :=
\cases{\displaystyle
\frac{(\gamma-\Lambda)^2}{2\sigma^2} ,&\quad if $\gamma\geq
\Lambda+ \sigma^2d$,\vspace*{4pt}\cr\displaystyle
d\biggl(\gamma- \Lambda- \frac{1}{2}\sigma^2d\biggr) ,&\quad if $\displaystyle
\Lambda+
\frac{1}{2}\sigma^2d \leq\gamma\leq\Lambda+ \sigma^2d$,\cr
\displaystyle
0 ,&\quad if $\displaystyle\gamma\leq\Lambda+ \frac{1}{2}\sigma^2d$.
}
\end{eqnarray*}
Then, for all $u > 0$, we have
\begin{eqnarray*}
\limsup_{T \to\infty} \frac{1}{T} \sup_{\mathcal{X}_T} \log
\mathbf{P}\Bigl(\sup_{x,y \in\mathcal{X}_T} \sup_{0 \leq t \leq T} |
\flow_t(x) - \flow_t(y)| \geq u\Bigr) \leq-I(\gamma),
\end{eqnarray*}
where $\sup_{\mathcal{X}_T}$ means that we take the supremum over all
cubes $\mathcal{X}_T$ in $\R^d$ with side length $\exp(-\gamma T)$.
\end{theorem}

\begin{pf}
The theorem can be proved via Kolmogorov's continuity theorem using
the explicit probabilistic upper bound for the modulus of continuity.
This proof and four others can be found in~\cite{sch09}, Chapter 2.3.
\end{pf}
\end{appendix}

\section*{Acknowledgments}
The author gratefully thanks Michael Scheutzow, Holger van Bargen and
Simon Wasserroth from TU Berlin for fruitful discussions.


%

\printaddresses


\begin{thebibliography}{17}

\bibitem{arn98}
%
\begin{bbook}[mr]
\bauthor{\bsnm{Arnold},~\bfnm{Ludwig}\binits{L.}}
(\byear{1998}).
\btitle{Random Dynamical Systems}.
\bpublisher{Springer}, \baddress{Berlin}.
\bid{mr={1723992}}
\bptok{imsref}%
\end{bbook}
%
\endbibitem

\bibitem{bax09}
%
\begin{barticle}[mr]
\bauthor{\bsnm{Baxendale},~\bfnm{Peter}\binits{P.}} \AND
\bauthor{\bsnm{Dimitroff},~\bfnm{Georgi}\binits{G.}}
(\byear{2009}).
\btitle{Uniform shrinking and expansion under isotropic {B}rownian flows}.
\bjournal{J. Theoret. Probab.}
\bvolume{22}
\bpages{620--639}.
\bid{doi={10.1007/s10959-008-0193-3}, issn={0894-9840}, mr={2530106}}
\bptok{imsref}%
\end{barticle}
%
\endbibitem

\bibitem{bax86}
%
\begin{barticle}[mr]
\bauthor{\bsnm{Baxendale},~\bfnm{Peter}\binits{P.}} \AND
\bauthor{\bsnm{Harris},~\bfnm{Theodore~E.}\binits{T.~E.}}
(\byear{1986}).
\btitle{Isotropic stochastic flows}.
\bjournal{Ann. Probab.}
\bvolume{14}
\bpages{1155--1179}.
\bid{issn={0091-1798}, mr={0866340}}
\bptok{imsref}%
\end{barticle}
%
\endbibitem

\bibitem{cc99}
%
\begin{bincollection}[mr]
\bauthor{\bsnm{Carmona},~\bfnm{Rene~A.}\binits{R.~A.}} \AND
\bauthor{\bsnm{Cerou},~\bfnm{Frederic}\binits{F.}}
(\byear{1999}).
\btitle{Transport by incompressible random velocity fields:
Simulations \&
mathematical conjectures}.
In \bbooktitle{Stochastic Partial Differential Equations: Six Perspectives}.
\bseries{Math. Surveys Monogr.}
\bvolume{64}
\bpages{153--181}.
\bpublisher{Amer. Math. Soc.}, \baddress{Providence, RI}.
\bid{mr={1661765}}
\bptok{imsref}%
\end{bincollection}
%
\endbibitem

\bibitem{css99}
%
\begin{barticle}[mr]
\bauthor{\bsnm{Cranston},~\bfnm{Michael}\binits{M.}},
\bauthor{\bsnm{Scheutzow},~\bfnm{Michael}\binits{M.}} \AND
\bauthor{\bsnm{Steinsaltz},~\bfnm{David}\binits{D.}}
(\byear{1999}).
\btitle{Linear expansion of isotropic {B}rownian flows}.
\bjournal{Electron. Commun. Probab.}
\bvolume{4}
\bpages{91--101}.
\bid{doi={10.1214/ECP.v4-1010}, issn={1083-589X}, mr={1741738}}
\bptok{imsref}%
\end{barticle}
%
\endbibitem

\bibitem{css00}
%
\begin{barticle}[mr]
\bauthor{\bsnm{Cranston},~\bfnm{Mike}\binits{M.}},
\bauthor{\bsnm{Scheutzow},~\bfnm{Michael}\binits{M.}} \AND
\bauthor{\bsnm{Steinsaltz},~\bfnm{David}\binits{D.}}
(\byear{2000}).
\btitle{Linear bounds for stochastic dispersion}.
\bjournal{Ann. Probab.}
\bvolume{28}
\bpages{1852--1869}.
\bid{doi={10.1214/aop/1019160510}, issn={0091-1798}, mr={1813845}}
\bptok{imsref}%
\end{barticle}
%
\endbibitem

\bibitem{dkk04}
%
\begin{barticle}[mr]
\bauthor{\bsnm{Dolgopyat},~\bfnm{Dmitry}\binits{D.}},
\bauthor{\bsnm{Kaloshin},~\bfnm{Vadim}\binits{V.}} \AND
\bauthor{\bsnm{Koralov},~\bfnm{Leonid}\binits{L.}}
(\byear{2004}).
\btitle{A limit shape theorem for periodic stochastic dispersion}.
\bjournal{Comm. Pure Appl. Math.}
\bvolume{57}
\bpages{1127--1158}.
\bid{doi={10.1002/cpa.20032}, issn={0010-3640}, mr={2059676}}
\bptok{imsref}%
\end{barticle}
%
\endbibitem

\bibitem{ito56}
%
\begin{binproceedings}[mr]
\bauthor{\bsnm{It{\^o}},~\bfnm{Kiyosi}\binits{K.}}
(\byear{1956}).
\btitle{Isotropic random current}.
In \bbooktitle{Proceedings of the {T}hird {B}erkeley {S}ymposium on
{M}athematical {S}tatistics and {P}robability, 1954--1955, Vol. {II}}
\bpages{125--132}.
\bpublisher{Univ. California Press}, \baddress{Berkeley}.
\bid{mr={0084890}}
\bptok{imsref}%
\end{binproceedings}
%
\endbibitem

\bibitem{kar91}
%
\begin{bbook}[mr]
\bauthor{\bsnm{Karatzas},~\bfnm{Ioannis}\binits{I.}} \AND
\bauthor{\bsnm{Shreve},~\bfnm{Steven~E.}\binits{S.~E.}}
(\byear{1991}).
\btitle{Brownian Motion and Stochastic Calculus},
\bedition{2nd} ed.
\bseries{Graduate Texts in Mathematics}
\bvolume{113}.
\bpublisher{Springer}, \baddress{New York}.
\bid{mr={1121940}}
\bptok{imsref}%
\end{bbook}
%
\endbibitem

\bibitem{kun90}
%
\begin{bbook}[mr]
\bauthor{\bsnm{Kunita},~\bfnm{Hiroshi}\binits{H.}}
(\byear{1990}).
\btitle{Stochastic Flows and Stochastic Differential Equations}.
\bseries{Cambridge Studies in Advanced Mathematics}
\bvolume{24}.
\bpublisher{Cambridge Univ. Press}, \baddress{Cambridge}.
\bid{mr={1070361}}
\bptok{imsref}%
\end{bbook}
%
\endbibitem

\bibitem{jan85}
%
\begin{barticle}[mr]
\bauthor{\bsnm{Le~Jan},~\bfnm{Yves}\binits{Y.}}
(\byear{1985}).
\btitle{On isotropic {B}rownian motions}.
\bjournal{Z. Wahrsch. Verw. Gebiete}
\bvolume{70}
\bpages{609--620}.
\bid{doi={10.1007/BF00531870}, issn={0044-3719}, mr={0807340}}
\bptok{imsref}%
\end{barticle}
%
\endbibitem

\bibitem{ls01}
%
\begin{barticle}[mr]
\bauthor{\bsnm{Lisei},~\bfnm{Hannelore}\binits{H.}} \AND
\bauthor{\bsnm{Scheutzow},~\bfnm{Michael}\binits{M.}}
(\byear{2001}).
\btitle{Linear bounds and {G}aussian tails in a stochastic dispersion model}.
\bjournal{Stoch. Dyn.}
\bvolume{1}
\bpages{389--403}.
\bid{doi={10.1142/S0219493701000175}, issn={0219-4937}, mr={1859014}}
\bptok{imsref}%
\end{barticle}
%
\endbibitem

\bibitem{sch09}
%
\begin{bincollection}[mr]
\bauthor{\bsnm{Scheutzow},~\bfnm{Michael}\binits{M.}}
(\byear{2009}).
\btitle{Chaining techniques and their application to stochastic flows}.
In \bbooktitle{Trends in Stochastic Analysis}.
\bseries{London Mathematical Society Lecture Note Series}
\bvolume{353}
\bpages{35--63}.
\bpublisher{Cambridge Univ. Press}, \baddress{Cambridge}.
\bid{mr={2562150}}
\bptok{imsref}%
\end{bincollection}
%
\endbibitem

\bibitem{ss02}
%
\begin{barticle}[mr]
\bauthor{\bsnm{Scheutzow},~\bfnm{Michael}\binits{M.}} \AND
\bauthor{\bsnm{Steinsaltz},~\bfnm{David}\binits{D.}}
(\byear{2002}).
\btitle{Chasing balls through martingale fields}.
\bjournal{Ann. Probab.}
\bvolume{30}
\bpages{2046--2080}.
\bid{doi={10.1214/aop/1039548381}, issn={0091-1798}, mr={1944015}}
\bptok{imsref}%
\end{barticle}
%
\endbibitem

\bibitem{vB10}
%
\begin{barticle}[mr]
\bauthor{\bparticle{van} \bsnm{Bargen},~\bfnm{Holger}\binits{H.}}
(\byear{2011}).
\btitle{A weak limit shape theorem for planar isotropic {B}rownian flows}.
\bjournal{Stoch. Dyn.}
\bvolume{11}
\bpages{593--626}.
\bid{doi={10.1142/S0219493711003474}, issn={0219-4937}, mr={2836563}}
\bptok{imsref}%
\end{barticle}
%
\endbibitem

\bibitem{yag57}
%
\begin{barticle}[author]
\bauthor{\bsnm{Yaglom},~\bfnm{A.~M.}\binits{A.~M.}}
(\byear{1957}).
\btitle{Some classes of random fields in $n$-dimensional space,
related to
stationary random processes}.
\bjournal{Theory Probab. Appl.}
\bvolume{2}
\bpages{273--320}.
\bptok{imsref}%
\end{barticle}
%
\endbibitem

\end{thebibliography}
\end{document}